\documentclass[11pt,a4paper]{article}
\usepackage{a4,amssymb,amsmath,amsthm,url,xcolor,enumerate,float}
\usepackage{bezier,amsfonts,amssymb,graphicx,amsthm,url}
\usepackage[english]{babel}
\title{Degree Monotone Paths and Graph Operations}
\date{}
\begin{document}
\newtheorem{theorem}{Theorem}[section]
\newtheorem{definition}{Definition}[section]
\newtheorem{proposition}[theorem]{Proposition}
\newtheorem{corollary}[theorem]{Corollary}
\newtheorem{lemma}[theorem]{Lemma}
\newcommand*\cartprod{\mbox{ } \Box \mbox{ }}
\newtheoremstyle{break}
  {}
  {}
  {\itshape}
  {}
  {\bfseries}
  {.}
  {\newline}
  {}

\theoremstyle{break}

\newtheorem{propskip}[theorem]{Proposition}
\DeclareGraphicsExtensions{.pdf,.png,.jpg}
\author{Yair Caro \\ Department of Mathematics\\ University of Haifa-Oranim \\ Israel \and Josef  Lauri\\ Department of Mathematics \\ University of Malta
\\ Malta \and Christina Zarb \\Department of Mathematics \\University of Malta \\Malta }

\maketitle
\begin{abstract}
A path $P$ in a graph $G$ is said to be a degree monotone path if the sequence of degrees of the vertices of $P$ in the order in which they appear on $P$ is monotonic.  The length of the longest degree monotone path in $G$ is denoted by $mp(G)$.  This parameter was first studied in an earlier paper by the authors where bounds in terms of other parameters of $G$ were obtained.

In this paper we concentrate on the study of how $mp(G)$ changes under various operations on $G$.  We first consider how $mp(G)$ changes when an edge is deleted, added, contracted or subdivided.  We similarly consider the effects of adding or deleting a vertex.  We sometimes restrict our attention to particular classes of graphs.  

Finally we study $mp(G \times H)$ in terms of $mp(G)$ and $mp(H)$ where $\times$ is either the Cartesian product or the join of two graphs.

In all these cases we give bounds on the parameter $mp$ of the modified graph in terms of the original graph or graphs and we show that all the bounds are sharp.
\end{abstract}
\section{Introduction}

Given a graph $G$, a degree monotone path is a path $v_1v_2 \ldots v_k$ such that $deg(v_1) \leq deg(v_2) \leq \ldots \leq deg(v_k)$ or $deg(v_1) \geq deg(v_2 \geq \ldots \geq deg(v_k)$.  This notion, inspired by the well-known Erd{\H{o}}s-Szekeres Theorem \cite{eliavs2013higher,erdos1935combinatorial}, was introduced in \cite{updownhilldom} under the name of uphill and downhill path in relation to domination problems, also studied in \cite{deering2013uphill,downhill2,downhill}.  In \cite{updownhilldom}, the authors specifically suggested the study of the parameter $mp(G)$, which denotes the length of the longest degree monotone path in $G$.  This parameter was first studied by the authors in \cite{dmpclz}.  Links between this parameter and other classical paramaters such as the chromatic number and clique number, using in particular the Gallai-Roy Theorem \cite{dougwest} were explored, and lower bounds and upper bounds for $mp(G)$ were established in \cite{dmpclz}.  The close relation to Turan numbers \cite{bollobas2004extremal} was also studied and explained in  \cite{dmpclz}.

In this paper we consider another natural question related to the parameter $mp(G)$, that of the effect of graph operations on this parameter.  We consider both  operations on a single graph $G$  which produce a new graph $G'$, as well as  operations applied to two graphs to produce a new single graph.

In the first section, we consider operations involving edges, namely edge addition and deletion, subdivision and edge contraction while in the second section vertex addition and deletion is discussed.  For each operation we obtain sharp bounds on $mp(G')$, and give constructions which achieve these bounds.  In some cases, we consider the operation for a particular family of graphs which gives more interesting results.

We then consider the \emph{Cartesian product} and the \emph{graph join} for two graphs $G$ and $H$, where again we give sharp bounds and constructions which achieve these bounds.

Any graph theory terms not defined here can be found in \cite{dougwest}.

\section{Edge Operations}

\subsection{Edge Addition and Deletion}

We now look at the concept of adding/deleting an edge to or from a given graph $G$, and consider the effect of these operations on $mp(G)$.  We add an edge $e$ by connecting two vertices in $V(G)$ which are non-adjacent, and the resulting graph is denoted by $G+e$, while when we delete an edge $e$, the resulting graph is denoted by $G-e$.

\begin{theorem}
Given a graph $G$,
\begin{enumerate}
\item{$\frac{mp(G)+1}{3} \leq mp(G+e) \leq 3mp(G)$.}
\item{$\frac{mp(G)}{3} \leq mp(G-e) \leq 3mp(G)-1$.}
\end{enumerate}

In both cases, the bounds attained are sharp.
\end{theorem}

\begin{proof}

Let us first consider the right hand sides of both inequalities.  

\noindent 1. \indent Let $G$ be a graph with $mp(G)=k \geq 1$.   If $k=1$ then $G$ is the empty graph with no edges, hence $mp(G+e)=2 \leq 3=3mp(G)$.  So let us assume that $k \geq 2$.  

Suppose that $mp(G+e)=t \geq 3k+1$.  Consider $P=v_1,v_2,\ldots,v_t$, a longest degree monotone path in $G+e$ in non-decreasing order.  Observe that $e$ must have at least one of its vertices in $P$, otherwise $P$ was originally in $G$.

\medskip

\noindent \emph{Case 1.1}. \indent Let us assume that $e=(v_i,z)$, where $z$ is not on the path.  If $i=1$, then clearly $v_1,\ldots,v_t$ is also a degree monotone path in $G$ and hence $t \leq mp(G)=k <3k$, a contradiction.

If $i=t$ then $v_1,\ldots,v_{t-1}$ is a degree monotone path in $G$ and hence $t-1 \leq k$ which implies $t \leq k+1 < 3k$.

So let us assume $1 < i <t$.  Let us consider the paths  $v_1,\ldots,v_{i-1}$ and $v_i,\ldots,v_t$ --- both are non-decreasing degree monotone paths in $G$ and hence together they have length at most $2k$, implying that $t \leq 2k <3k$.

\medskip

\noindent \emph{Case 1.2}. \indent Now assume $e=(v_i,v_j)$, $i<j$.  Hence, in $G+e$, all vertices in $P$ have the same degree as they have in $G$ except for $v_i$ and $v_j$ whose degree has increased by one.  Consider the paths $P_1=v_1,\ldots,v_i$, $P_2=v_i,\ldots,v_{j-1}$ and $P_3=v_j,\ldots,v_t$.

$P_3$ is  clearly a  degree monotone path in $G$ since $deg_G(v_j)+1 = deg_{G+e}(v_j) \leq deg_{G+e}(v_{j+1})$ ( or $j=t$), hence $|V(P_3)| \leq k$.  Similarly $|V(P_2)| \leq k$.  So if $t \geq 3k+1$, $|V(P_1)| \geq k+2$ since $v_i$ is both in $P_1$ and $P_2$.  But $v_1,\ldots,v_{i-1}$ is a degree monotone path of length $k+1$ in $G$, contradicting the fact the $mp(G)=k$.  Hence $mp(G+e) \leq 3k$.

\medskip

For sharpness of the bound, consider $P_{3k}$, to which we add a leaf to vertices $v_2$ up to $v_{3k-2}$ except vertices $v_{k+1}$ and $v_{2k+1}$, and we connect $v_{3k-1}$ to a new vertex $z$ to which we add two leaves.  Thus the resulting graph $G_1^+$ has  $6k-2$ vertices.   Figure \ref{diag1} shows the construction for $k=4$. It is clear that $mp(G_1^+)=k$.  Now if we add the edge $e=(v_{k+1},v_{2k+1})$, these two vertices now have degree 3 also, and hence the path $v_1,v_2,\ldots,v_{3k-1},z$ is a degree monotone path so $mp(G_1^++e)=3k=3mp(G_1^+)$.

\begin{figure}[h!]
\centering
\includegraphics{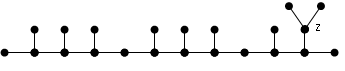} 
\caption{$G_1^+$  when $k=4$} \label{diag1}
\end{figure}

\noindent 2. \indent  Let $G$ be a graph with $mp(G)=k \geq 1$.   If $k=1$ then $G$ is the empty graph with no edges, hence we cannot delete edges.  Therefore assume $k \geq 2$ and suppose that $mp(G-e)=t \geq 3k$.  Let $P=v_1,\ldots,v_t$ be a degree monotone path of maximum length in non-increasing order.  Observe that $e$ must have at least one of its vertices in $P$, otherwise $P$ was originally in $G$.

\medskip

\noindent \emph{Case 2.1}. \indent  Let us first assume that $e=(v_i,z)$, where $z$ is not on the path.   If $i=1$, then clearly $v_1,\ldots,v_t$ is also a degree monotone path in $G$ and hence $t \leq mp(G)=k <3k$, a contradiction.

If $i=t$ then $v_1,\ldots,v_{t-1}$ is a degree monotone path in $G$ and hence $t-1 \leq k$ which implies $t \leq k+1 < 3k$.

So let us assume $1 < i <t$.  Let us consider the paths  $v_1,\ldots,v_{i-1}$ and $v_i,\ldots,v_t$ --- both are non-increasing degree monotone paths in $G$ and hence together they have length at most 2k, implying that $t \leq 2k <3k$.
\medskip

\noindent \emph{Case 2.2}. \indent Now assume $e=(v_i,v_j)$, $i<j$.  Hence, in $G-e$, all vertices in $P$ have the same degree as they have in $G$ except for $v_i$ and $v_j$ whose degree has decreased by one.  Consider the paths $P_1=v_1,\ldots,v_i,v_j$, $P_2=v_i,\ldots,v_{j-1}$ and $P_3=v_i,v_j,\ldots,v_t$.

$P_3$ is  clearly a  degree monotone path in $G$ since $v_i$ and $v_j$ have a larger (by 1) degree in $G$ and $deg(v_i) \geq deg(v_j)$.  Similarly $|V(P_2)| \leq k$.  So if $t \geq 3k$, $|V(P_1)| \geq k+3$ since $v_i$ and $v_j$ are also on $P_2$ and $P_3$.  But $v_1,\ldots,v_{i-1}$ is a degree monotone path of length $k+1$ in $G$, contradicting the fact the $mp(G)=k$.  Hence $mp(G-e) \leq 3k-1$.

\medskip

For sharpness, let us construct the graph $G_1^-$ as follows: we start with $P_{3k-1}$ and add three leaves to $v_1$and a leaf to $v_2$;  we then connect $v_{k+1}$ to $v_{2k+1}$ --- we call this edge $e$.  Figure \ref{diag2} shows the construction for $k=4$.  Clearly $mp(G_1^-)=k$ while $mp(G_1^- -e)=3k-1$ with the path $v_1,\ldots,v_{3k-1}$ and hence  $mp(G_1^--e)=3mp(G_1^-)-1$.

\begin{figure}[h!]
\centering
\includegraphics{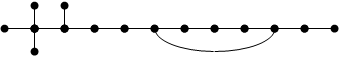} 
\caption{$G_1^-$  when $k=4$} \label{diag2}
\end{figure}

\bigskip

Let us now turn to the lower bounds in both cases.  These results are derived from those for the upperbounds.
\medskip

\noindent 1. \indent  Suppose for some graph $G$, $mp(G+e) < \frac{mp(G)+1}{3}$.  Then if we take $H=G+e$ so that $G=H-e$, we have $mp(H) < \frac{mp(H-e)+1}{3}$, which implies that $mp(H-e) > 3mp(H)-1$, contradicting the righthand side of part 2.

For sharpness, let $G_2^{+}$ be the graph $G_1^-$ without the edge $e$.  Then $mp(G_2^{+}+e)=k$ while $mp(G_2^{+})=3k-1$ giving $mp(G_2^{+}+e)=\frac{mp(G_2^{+})+1}{3}$.

\medskip

\noindent 2. \indent  Suppose that for some graph $G$, $mp(G-e) < \frac{mp(G)}{3}$.  Then let $H=G-e$ so that $G=H+e$ giving $mp(H) < \frac{mp(H+e)}{3}$ which implies $mp(H+e) > 3mp(H)$, contradicing the upper bound in part 1.

For sharpness of the bound, let $G_2^{-}$ be the graph $G_1^+$ with the added edge $e=(v_{k+1},v_{2k+1})$.  Then $mp(G_2^{-})=3k$ and $mp(G_2^{-}-e)=\frac{mp(G_2^{-})}{3}$.
\end{proof}

\subsection{Subdivision}

Given a graph $G$ and $e=(u,v) \in E(G)$, the subdivision of $e$ is the addition of a new vertex $w$ such that $(u,w)$ is an edge and $(w,v)$ is an edge but $(u,v)$ is no longer an edge.  Note that the degree of $u$ and $v$ remains the same, and the degree of $w$ is 2.  The graph obtained by subdividing $e$ is denoted by $G^*$.  Again, we look at the effect of this operation on the maximum length of a degree monotone path, $mp(G)$.

\begin{theorem}
Let $G$ be a graph and $e=(u,v)$ an edge in $G$ such that $G^*$ is the graph obtained by subdividing $e$ with the vertex $w$. Then \[\left \lceil \frac{mp(G)+1}{2} \right \rceil \leq mp(G^*) \leq mp(G)+1,\] and both bounds are sharp.
\end{theorem}

\begin{proof}
Let us first consider the upper bound.  Let $P$ be a degree monotone path of maximum length in $G^*$.  We consider the following cases:

\noindent \emph{Case 1.} \indent If $w \not \in P$, then $P$ is a degree monotone path in $G$ hence $mp(G^*) \leq mp(G)$.

\noindent \emph{Case 2.} \indent  If $w \in P$, then either $u$ or $v$ or both are in $P$ otherwise $P=\{w\}$ which is not maximal.  We consider these cases separately:
\begin{itemize}
\item{If both $u$ and $v$ are in $P$ then $P-\{w\} \cup (u,v)$ is a degree monotone path in $G$ hence $mp(G^*) \leq mp(G)+1$.}
\item{If only one of $u$ or $v$, say $u$, is in $P$ then $P-\{w\}$ is a degree monotone path in $G$ and again $mp(G^*) \leq mp(G)+1$.}
\end{itemize}

For sharpness of the bound, consider $G=P_n$ the path on $n$ vertices --- then $mp(G)=n-1$, and subdividing any edge gives $mp(G^*)=n=mp(G)+1$.

For the lower bound, let $P$ be a degree monotone path of maximum length in $G$ and let $P=v_1,v_2,\ldots,v_t,v_{t+1},\ldots,v_k$, such that $u=v_t$ and $v=v_{t+1}$.  Let us consider $e=(u,v)$ and the subdividing vertex $w$.

\noindent \emph{Case 1.} \indent If $e=(u,v)$ is not on the path $P$, then $P$ is a degree monotone path in $G^*$ hence $mp(G^*) \geq mp(G)$.

\noindent \emph{Case 2.} \indent If the edge $e=(u,v)$ is in $P$, then as $w$ subdivides $e$, we get $P^*=v_1,\ldots,v_t,w,v_{t+1},\ldots,v_k$ a path in $G^*$, where $v_t=u$ and $v_{t+1}=v$. 

We assume without loss of generality that $P$ is a non-decreasing monotone path, and hence $deg(u) \leq deg(v)$.  The vertex $w$ has degree 2.  Now if $deg(u)>2$,  the paths $v_1,\ldots,v_t$ and the paths $w,v_{t+1},\ldots,v_k$ are degree monotone in $G^*$, while if $deg(u) \leq 2$, the path $P^*$ is degree montone in $G^*$.  Hence in $G^*$ there is a degree monotone path of length at least $\left \lceil \frac{k+1}{2} \right \rceil = \left \lceil \frac{mp(G)+1}{2} \right \rceil$.

For sharpness of the lower bound, let us take the path $P_n$ and add a leaf to the vertices $v_2$ up to $v_{n-1}$.  Hence $mp(G)=n-1$.  If we subdivide the edge $(v_{\lceil \frac{n}{2} \rceil},v_{\lceil \frac{n}{2} \rceil+1})$, we get $mp(G^*)=\left \lceil \frac{n}{2} \right \rceil= \left \lceil \frac{mp(G)+1}{2} \right \rceil$.
\end{proof}

\subsection{Edge Contraction}
In a graph $G$, \emph{contraction of an edg}e $e=(u,v)$ is the replacement of $u$ and $v$ with a single vertex $w$ adjacent (without multiple edges) to all vertices in $N(u) \backslash v \cup N(v) \backslash u$.  The resulting graph $G \cdot e$ has one less  vertex  than $G$.    In case a vertex $z \in V(G)$ is adjacent to both $u$ and $v$, the degree of $z$ in $G \cdot e$ decreases by one --- otherwise it remains the same as in $G$.  In view of this we first consider triangle-free graphs, in which case the degrees of the neighbours of $u$ and $v$ remain unchanged in $G \cdot e$, and $deg(w)=deg(u)+deg(v)-2$.
\begin{theorem}
Let $G$ be a triangle-free graph.  Let $e=uv$ be an edge of $G$ which is contracted to form $G \cdot e$ with new vertex $w$.  Then \[\frac{mp(G)}{3} \leq mp(G \cdot e) \leq 2mp(G).\]
\end{theorem}

\begin{proof}

Let us first consider the upper bound.  Clearly $mp(G) \geq 2$ as if $mp(G)=1$, $G$ has no edges and $G \cdot e$ is not defined.   Let $P=v_1v_2 \ldots v_k$ be a degree monotone path in non-decreasing order of maximum length in $G \cdot e$.  We know that $deg(w)=deg(u)+deg(v)-2$. Let us look at all the different possibilities.
\medskip

\noindent \emph{Case 1}: \indent Assume first that $deg(u)=deg(v)=1$ and hence $deg(w)=0$.  If this was the only edge in $G$, then $mp(G)=2$ while $mp(G \cdot e)=1$, and the upper bound holds.  If there are other edges in $G$, then $mp(G \cdot e)=mp(G)$, and again the upper bound holds.

\medskip

\noindent \emph{Case 2}: \indent Now assume $1 = deg(u) < deg(v)$ --- then $deg(w)=deg(v)-1$ in $G \cdot e$.  We consider the following cases:
\begin{itemize}
\item{Clearly, if $w$ is not a vertex in $P$, then $P$ is degree monotone in $G$ too, hence $mp(G \cdot e) \leq mp(G)$.}
\item{ If $w=v_1$ in $P$, then in $G$, $v_2 \ldots v_k$ is a degree monotone path, and hence $mp(G \cdot e) - 1 \leq mp(G)$ and hence $mp(G \cdot e) \leq mp(G)+1 \leq 2mp(G)$.}
\item{ If $w=v_k$, then $v_1 \ldots v_{k-1}v$ is degree monotone in $G$ and hence $mp(G \cdot e)  \leq mp(G)$.}

\item{If $w=v_j$ for $2 \leq j \leq k-1$, then $v_1 \ldots v_{j-1}v$  and $v_{j+1} \ldots v_k$ are degree monotone in $G$, and $\min\{\max\{j,k-j\}\}= \lceil \frac{k}{2} \rceil$.

If $k$ is odd, then $mp(G) \geq \frac{k+1}{2}$ and hence $mp(G \cdot e) \leq 2mp(G)-1$.

If $k$ is even, then $mp(G) \geq \frac{k}{2}$ and hence $mp(G \cdot e) \leq 2mp(G)$.}
\end{itemize}

\noindent \emph{Case 3}: \indent It remains to consider, without loss of generality, the case \[2 \leq deg(u) \leq deg(v) \leq deg(w)=deg(u)+deg(v)-2.\]  Again we consider each possibility.

\begin{itemize}
\item{If in $G$, $w$ is not a vertex in $P$, then $P$ is degree monotone in $G$ too, hence $mp(G \cdot e) \leq mp(G)$.}
\item{If $w=v_1$, then in $G$, either $u$ or $v$ is adjacent to $v_2$ --- since $deg(u) \leq deg(v) \leq deg(w)$ then either $uv_2 \ldots v_k$ or $vv_2 \ldots v_k$ is a degree monotone path in $G$, and hence $mp(G \cdot e) + 1 \leq mp(G)$ giving $mp(G \cdot e) \leq mp(G) -1 \leq 2mp(G)$.}
\item{If $w=v_k$, then $v_1 v_2 \ldots v_{k-1}$ is still a degree monotone path in $G$ hence $mp(G \cdot e) -1 \leq mp(G)$ that is $mp(G \cdot e) \leq mp(G) + 1 \leq 2mp(G)$.}
\item{If $w=v_j$, $2 \leq j \leq k-1$, then let us consider the following cases:
\begin{enumerate}
\item{If in $G$, $u$ is adjacent to both $v_{j-1}$ and $v_{j+1}$, and hence $v$ is not adjacent to either of these vertices since $G$ is triangle-free, then $v_1 \ldots v_{j-1}$ and $u,v_{j+1} \ldots v_k$ are degree monotone paths of length $j-1$ and $k-j+1$ respectively in $G$.  Again, $\min\{\max\{j-1,k-j+1\}\}= \lceil \frac{k}{2} \rceil$.

If $k$ is odd, then $mp(G) \geq \frac{k+1}{2}$ and hence $mp(G \cdot e) \leq 2mp(G)-1$.

If $k$ is even, then $mp(G) \geq \frac{k}{2}$ and hence $mp(G \cdot e) \leq 2mp(G)$.}

\item{ If in $G$, $u$ is adjacent to $v_{j-1}$ and $v$ is adjacent to $v_{j+1}$, then if $deg(v_{j-1}) \leq deg(u)$ then $v_1v_2,\ldots,v_{j-1},u,v,v_{j+1},\ldots v_k$ is degree monotone in $G$ and hence $mp(G \cdot e)+1 \leq mp(G)$ implying that $mp(G \cdot e) \leq mp(G)-1 \leq 2mp(G)$.

 If, on the other hand, $deg(v_{j-1}) > deg(u)$, then $deg(u) \leq deg(v) \leq deg(w) \leq deg(v_{j+1})$ and hence $v_1 v_2 \ldots v_{j-1}$  and $u,v,v_{j+1}\ldots v_k$ are  degree monotone paths in $G$ of length $j-1$ and $k-j+2$ respectively.   In this case we consider $\min\{\max\{j-1,k-j+2\}\}$. 

 If $k$ is even this is equal to $\frac{k+2}{2}$, and hence $mp(G \cdot e) \leq 2mp(G)-2$. 

 If  $k$ is odd, we get $\frac{k+1}{2}$, and hence $mp(G \cdot e) \leq 2mp(G)-1$.}
\item{If in $G$, $v$ is adjacent to $v_{j-1}$ and $u$ is adjacent to $v_{j+1}$, then $v_1v_2 \ldots v_{j-1}$ and $u v_{j+1} \ldots v_k$ are degree monotone paths of length $j-1$ and $k-j+1$ respectively in $G$, and again $\min\{\max\{j-1,k-j+1\}\}= \lceil \frac{k}{2} \rceil$.

If $k$ is odd, then $mp(G) \geq \frac{k+1}{2}$ and hence $mp(G \cdot e) \leq 2mp(G)-1$.

If $k$ is even, then $mp(G) \geq \frac{k}{2}$ and hence $mp(G \cdot e) \leq 2mp(G)$.}

\item{If in $G$, $v$ is adjacent to both $v_{j-1}$ and $v_{j+1}$, and $u$ is adjacent to neither since $G$ is triangle-free, then $v_1 \ldots v_{j-1}$ and $u,v,v_{j+1} \ldots v_k$ are  degree monotone paths in $G$ of length $j-1$ and $k-j+2$ respectively.  In this case we consider $\min\{\max\{j-1,k-j+2\}\}$. 

 If $k$ is even this is equal to $\frac{k+2}{2}$, and hence $mp(G \cdot e) \leq 2mp(G)-2 $. 

 If  $k$ is odd, we get $\frac{k+1}{2}$, and hence $mp(G \cdot e) \leq 2mp(G)-1 $.}
\end{enumerate}}
\end{itemize}
This bound is attained by the graph $G_1$ constructed as follows:  consider the path on $2k+1$ vertices --- we add a leaf to vertices $v_2$ up to $v_{2k}$, and to the vertices $v_k$ and $v_{2k}$ we add a second leaf.  Then $mp(G)=k$.  If we contract one of the edges joining $v_k$ and a leaf, then vertex $w$ has degree $3$ and hence $v_1v_2 \ldots v_{k-1}w v_{k+2}\ldots v_{2k}$ is degree monotone in $G \cdot e$ and has length $2k$, giving $mp(G \cdot e) = 2mp(G)$.  Figure \ref{contract1} shows the construction for $k=4$.

\begin{figure}
\centering
\includegraphics{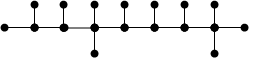} 
\caption{$G_1$  when $k=4$} \label{contract1}
\end{figure}

\bigskip

We now consider the lower bound.  Let $P=v_1 \ldots v_k$ be a degree monotone path in non-decreasing order of maximum length in $G$.  Let $e=uv$ be the edge contracted to vertex $w$ in $G \cdot e$, and without loss of generality, we assume $deg(u) \leq deg(v)$.  Clearly, if the vertices $u$ and $v$ are not on $P$, then $P$ is still degree monotone in $G \cdot e$ and hence $mp(G \cdot e) \geq mp(G)$.  So let us assume that $u$ and $v$ are on $P$, and hence $k=mp(G) \geq 2$.  If $k=2$, that is $P=uv$, then $mp(G \cdot e)\geq 1 \geq \frac{mp(G)}{2}$.  So let us assume that $k \geq 3$ and hence $deg(v) \geq 2$.  So let us assume that $u$ and $v$ are on $P$ and $mp(G) \geq 3$.
\medskip

\noindent \emph{Case 1.} \indent We first consider the case  in which either $u$ or $v$, but not both, are on $P$.  Without loss of generality, Let $u=v_j$ be on $P$.  Then in $G \cdot e$, $w$ is on $P$, and $v_1 \ldots v_{j-1}$ and $v_{j+1} \ldots v_k$ are degree montone in $G \cdot e$, of length $j$ and $k-j$ respectively.  Now $\min\{\max\{j,k-j\}\}=\lceil \frac{k}{2} \rceil$. 

 If $k$ is even then $mp(G \cdot e) \geq \frac{mp(G)}{2} \geq \frac{mp(G)}{3}$.

If $k$ is odd then $mp(G \cdot e) \geq \frac{mp(G)+1}{2} \geq \frac{mp(G)}{3}$.
\medskip

\noindent \emph{Case 2.} \indent We now consider the case in which $u$ and $v$ are in $P$ but $e=uv$ is not in $P$.  Let $u=v_i$ and $v=v_j$ in $P$; in all cases $j-i \geq 3$, otherwise we have a copy of $K_3$ in $G$ --- we consider the following cases:
\begin{itemize}
\item{If $i=1$ and $j=k$, then $v_2 \ldots v_{k-1}w$ is degree monotone in $G \cdot e$, hence $mp(G \cdot e) \geq mp(G)-1 \geq \frac{mp(G)}{3}$ since $mp(G) \geq 2$.}
\item{If $i=1$ and $4 \leq j < k$, then $v_2 \ldots v_{j-1}w$ and $v_{j+1} \ldots v_k$ are degree monotone in $G \cdot e$.  Hence we need $\min\{\max\{j-1, k-j\}\}= \lfloor \frac{k}{2} \rfloor$. Then $mp(G \cdot e) \geq \lfloor \frac{mp(G)}{2} \rfloor \geq \frac{mp(G)}{3}$.}
\item{If $i>1$ and $j=k$, then $v_1 \ldots v_{i-1}$ and $v_{i+1} \ldots v_{k-1}w$ are degree monotone in $G \cdot e$.  Hence we need $\min\{\max\{i-1, k-i\}\}=\lfloor \frac{k}{2} \rfloor$. Then $mp(G \cdot e) \geq \lfloor \frac{mp(G)}{2} \rfloor \geq \frac{mp(G)}{3}$.}

\item{If $1 < i < j <k$, then $v_1 \ldots v_{i-1}w$, $v_{i+1} \ldots v_{j-1}w$ and $v_{j+1} \ldots v_k$ are degree monotone in $G \cdot e$, of lengths $i$, $j-1$ and $k-j$ respectively.  Let $k=3t+r$ where $0 \leq r \leq 2$. 

 If $k-j > t$, then $k-j \geq t+1 \geq \frac{k}{3}$, and hence $mp(G \cdot e) \geq \frac{mp(G)}{3}$.

So let us assume that $k-j \leq t$ and hence $j \geq k-t=2t+r$.  Now $\max \{i,j-1\} \geq \frac{2t+r}{2}=t + \frac{r}{2}$.  If $r=0$ then $k \geq 3t$ hence $\max \{i,j-1\} = t = \frac{k}{3}$.  If $r=1$, then $k=3t+1$, then $\max \{i,j-1\} \geq t+1 = \frac{k+2}{3} > \frac{k}{3}$ since $\max \{i,j-1\}$ must be an integer.  Finally, if $r=2$, then $\max \{i,j-1\} \geq t+1 = \frac{k+1}{3} > \frac{k}{3}$.

Hence, in each case, $mp(G \cdot e) \geq \frac{mp(G)}{3}$.}
\end{itemize}

\noindent \emph{Case 3.} \indent
Finally, we consider the case in which $e=uv$ is in $P$.  We may assume that $mp(G) \geq 3$ since if$mp(G)=2$ then clearly $mp(G \cdot e) \geq 1 \geq \frac{mp(G)}{3}$.  This results in the following cases:
\begin{itemize}
\item{If $u=v_1$ and $v=v_2$, then in $G \cdot e$, $v_3 \ldots v_k$ is degree monotone, and hence $mp(G \cdot e) \geq mp(G) -2 \geq \frac{mp(G)}{3}$ for $mp(G) \geq 3$.}
\item{If $u=v_{k-1}$ and $v=v_k$, hence $2 \leq deg(u) \leq deg(v) \leq deg(w)$, then in $G \cdot e$, $v_1 \ldots v_{k-2}w$ is degree monotone since $deg(u) \leq deg(v) \leq deg(w)=deg(u)+deg(v)-2$.  Hence $mp(G \cdot e) \geq mp(G)-1 \geq \frac{mp(G)}{3}$ for $mp(G)\geq 3$.}
\item{If  $u=v_j$ and $v=v_{j+1}$ for $2 \leq j \leq k-2$ (so $k \geq 4$ otherwise we have one of the previous two cases),  then in $G \cdot e$, $v_1 \ldots v_{j-1}w$ and $v_{j+2} \ldots v_k$ are degree monone paths of lengths $j$ and $k-j-1$ respectively--- hence we must consider $\min\{\max\{j, k-j-1\}\}= \lceil \frac{k}{2} \rceil$.  Hence $mp(G \cdot e) \geq \lceil \frac{k}{2} \rceil \geq \frac{mp(G)}{3}$ as required.}

\end{itemize}

The bound is attained by the following construction to give graph $G_3$:  let us take the path on $3k+3$ vertices --- we add a leaf to the vertices $v_2$ up to $v_{3k+2}$ except for $v_{k+1}$ and $v_{2k+2}$, and we connect these two vertices by an edge $e$. We also attach three leaves to each of the leaves attached to $v_{k+2}$ and $v_{2k+1}$.  Finally we add 3 leaves to vertex $v_{3k+3}$.  Thus $deg(v_1)=1$, $deg(v_2)=deg(v_3) \ldots =deg(v_{3k+2})=3$ and $deg(v_{3k+3})=4$, hence $mp(G)=3k+3$.  Now let us contract the edge $e$.  Then  $deg(v_1)=1$, $deg(v_2)=deg(v_3) \ldots =deg(v_k)=deg(v_{k+2}) = \ldots = deg(v_{2k+1})=deg(v_{2k+3})=\ldots =deg(v_{3k+2})=3$ and $deg(v_{3k+3})=4$.  Vertex $w$ has degree 4.  Hence the possible degree monotone paths in $G \cdot e$ are $v_1 \ldots v_k w$, $v_{k+2} \ldots v_{2k+1}w$ and $v_{2k+3} \ldots v_{3k+3}$ --- all these are of length $k+1$ and therefore $mp(G_3 \cdot e)=\frac{mp(G_3)}{3}$.  Figure \ref{contract3} shows the construction for $k=3$.

\end{proof}

\begin{figure}[h!]
\centering
\includegraphics{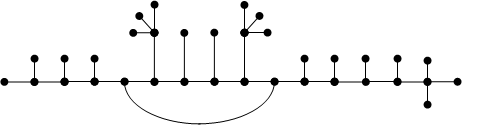} 
\caption{$G_3$  when $k=4$} \label{contract3}
\end{figure}

The following proposition shows that if we consider contraction in graphs which are not triangle-free, the situation is very different.

\begin{theorem}

There exist arbitrarily large $K_4$-free graphs $G$ on $n$ vertices such that $mp(G)=4$ while $mp(G.e)=|V(G \cdot e)| =n-1$.
\end{theorem}

\begin{proof}
Consider the following construction:  $G$ is the graph constructed by taking a path $P=v_1 \ldots v_{4k}$ on $4k$ vertices where $k \geq 2$, and another edge $e=uv$.  We connect the vertices  $v_{2i} \in P$, $1 \leq i \leq 2k$ to both $u$ and $v$ --- we then connect $v_{2i-1} \in P$ for $1 \leq i \leq k$ to $u$, and  $v_{2i-1} \in P$ for $k+1 \leq i \leq 2k$ to $v$.  Finally we connect $v_1$ to $v_{4k}$.  Figure \ref{contract2} gives the construction for $k=2$.  $G$ is not triangle-free, but it is $K_4$-free.  One can see that in $P$, the vertices with even index have degree 4, while the vertices with odd index have degree 3, and $deg(u)=deg(v)=3k$.  Hence it is clear that $mp(G)=4$ for any value of $k \geq 2$.  Now if we contract $e=uv$ to a vertex $w$, all vertices in $P$  have degree 3, while $deg(w)=4k \geq 3$, and therefore $mp(G \cdot e) = 4k+1$.
\end{proof}
\begin{figure}[h!]
\centering
\includegraphics{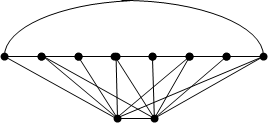} 
\caption{$G$  when $k=3$} \label{contract2}
\end{figure}

\section{Vertex Addition and Deletion}

Given a graph $G$, we can add a vertex $v$ and connect it to at least one vertex in $G$ to give the graph $G+v$.  On the other hand, given a graph $G$ and a vertex $v \in V(G)$, the graph $G-v$ is obtained by deleting the vertex $v$ and all its incident edges.  The effect on the maximum length of a degree monotone path can be seen through the following observations.

\begin{proposition}
Given a graph $G$ and $v \in V(G)$,
\begin{enumerate}
\item{$2\leq mp(G+v) \leq |V(G)|+1$}
\item{$1 \leq mp(G-v) \leq |V(G)|-1$}
\end{enumerate}
\end{proposition}
\begin{proof}
\mbox{ \\}

\noindent 1. \indent Consider the complete bipartite graph $G=K_{n,n+1}$ so that $mp(G)=2$.  We add a vertex $v_1$ and connect it to all the vertices in the larger part, to give $H=G+v_1=K_{n+1,n+1}$ --- now $mp(G+v_1)=mp(H)=|V(G)|+1$.  Now if we add another vertex $v_2$ and connect it to all the vertices in one part of the partition, we get $H+v_2=F=K_{n+1,n+2}$ and again $mp(H+v_2)=mp(F)=2$.
\medskip

\noindent 2. \indent For the upperbound, consider the graph $G=K_n$ --- then $mp(G)=n$, and $G-v=K_{n-1}$ hence $mp(G-v)=n-1=|V(G)|-1$.

For the lower bound consider $G=K_{1,m}$, $m \geq 1$.  Then deleting the vertex of degree $m$ gives $mp(K_{1,m}-v)=mp(G-v)=1$.
\end{proof}

In view of this general result, we consider vertex addition and deletion for the family of trees.  The following example shows that if one adds non-leaf vertices, the effect on  $mp(G)$ can be quite drastic.  Consider the tree $T$ constructed by taking a path on $2k+1$ vertices $(v_1 v_2 \ldots v_{2k+1})$, $k \geq 1$, and adding a leaf to the vertices $v_{2i}$, $1 \leq i \leq k$.  Clearly this tree has $mp(T)=2$.  Let us add a vertex $v$ and connect it to the vertices $v_{2i+1}$, $0 \leq i \leq k$.  So now all the vertices on the path have degree 3 except for $v_1$ and $v_{2k+1}$, and vertex $v$ has degree $k+1 \geq 2$.  Hence there is now a degree monotone path of length $2k+1$, that is $mp(T+v)=2k+1$.  Hence, in the sequel, we only consider the addition of leaves, so that the resulting graph is another tree.

\begin{theorem}
Let $T$ be a tree and $v \in V(T)$,
\begin{enumerate}
\item{If we add a vertex $v$ such that $T+v$ is also a tree, then \[\frac{mp(T)}{2} \leq mp(T+v) \leq 2mp(T).\]}
\item{If $v \in V(T)$ such that $T-v$ is a tree, then \[\frac{mp(T)}{2} \leq mp(T-v) \leq 2mp(T).\]}
\end{enumerate}

In both cases, the bounds attained are sharp.
\end{theorem}

\begin{proof}

Let us first consider the upper bound for each case.

\noindent 1. \indent If $T+v$ is a tree then clearly $v$ is a leaf otherwise it would form a cycle.  Let $P=v_1,v_2,\ldots,v_k$ be a degree monotone path of maximum length in $T+v$ in non-decreasing order.  Observe that $k \geq 2$ as $T+v$ contains at least one edge.  So we consider all possible cases.

\begin{itemize}
\item{If $v$ is not adjacent to any vertex in $P$ then $P$ is a degree monotone path in $T$ hence $mp(T+v) \leq mp(T)$.}
\item{If $v$ is adjacent to $v_1$, then $v,v_1,\ldots,v_k$ is a degree monotone path in $T+v$ contradicting the maximality of $P$.}
\item{If $v=v_1$ then $v_2,\ldots,v_k$ is a degree monotone path in $T$ of length $k-1$ so $k-1 \leq mp(T)$ and hence $2mp(t) \geq 2k-2 \geq k=mp(T+v)$ since $k \geq 2$.}
\item{If $v$ is adjacent to $v_k$, then $v_1,\ldots,v_{k-1}$ is a degree monotne path in $T$, hence again $2mp(t) \geq 2k-2 \geq k=mp(T+v)$ since $k \geq 2$.}
\item{If $v=v_k$, this implies that $deg(v)=1$ and that $k=2$, and since $T$ is not $K_1$, $mp(T) \geq 2$ hence $2mp(T) \geq 4 > 2=k=mp(T+v)$. (If $T=K_1$ then trivially $T+v=K_2$ so $mp(T+v)=2 =2 mp(T)$.)}
\item{ Let $v$ be adjacent to some $v_j$ for $2 \leq j \leq k-1$.  Note that $v$ can only be  $v_1$ or $v_k$ in $P$ since it is a leaf.  Consider $v_1,\ldots,v_{j-1}$ which is a degree monotone path in $T$ hence $j-1 \leq mp(T)$.  Similarly $v_j,\ldots,v_k$ is also a degree monotne path in $T$ with length at most $mp(T)$ hence by adding we get $mp(T+v) \leq 2 mp(T)$ as required.}
\end{itemize}

For sharpness, consider the tree $T_1^+$ constructed as follows:  take the path on $2k+1$ vertices $v_1,\ldots,v_{2k+1}$, where $k \geq 2$ --- we add a leaf to all the vertices except $v_1$,$v_{k+1}$ and $v_{2k+1}$.  Then $mp(T_1^+)=k$.  If we add a vertex $v$ and connect it to  vertex $v_{k+1}$, then $mp(T_1^++v)=2k=2mp(T_1^+)$.  Figure \ref{diag3} shows the construction for $k=4$.
\medskip

\begin{figure}[h!]
\centering
\includegraphics{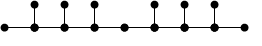} 
\caption{$T_1^+$  when $k=4$} \label{diag3}
\end{figure}
\noindent 2. \indent  Let $P=v_1,v_2,\ldots,v_k$ be a degree monotone path of maximum length in $T-v$ in non-decreasing order.  Not that if $k=1$ then $T-v$ has no edges which imples that $T=K_{1,m}$ for some $m \geq 1$, and so $mp(T)=2=2mp(T-v)$.

So let us assume that $k \geq 2$.  If in $T$ $v$ is not adjacent to any vertex of $P$, then $P$ is also a degree monotone path in $T$, hence $mp(T-v)=k \leq mp(T) < 2mp(T)$.

So we consider the case in which $v$ is adjacent to exactly one vertex in $P$.  We consider the different scenarios.

\begin{itemize}
\item{If $v$ is adjacent to $v_1$ in $T$, then $v_2,\ldots,v_k$ is a degree monotone path of length $k-1$ in $T$ hence $mp(T-v)=k \leq 2(k-1) \leq 2mp(T)$ since $k \geq 2$.}
\item{If $v$ is adjacent to $v_k$ in $T$ then $v_1,\ldots,v_k$ is also a degree monotone path in $T$,  hence again $mp(T-v)=k \leq  mp(T) < 2mp(T)$.}
\item{If $v$ is adjacent to some $v_j$ for $2 \leq j \leq k-1$, then consider in $T$ the path $v_1,\ldots,v_j$ --- this is degree monotone hence $j \leq mp(T)$.  Now $v_{j+1},\ldots,v_k$ is also a degree monotone path in $T$ and has length at most $mp(T)$.  Adding, we get $mp(T-v) \leq 2 mp(T)$ as required.}
\end{itemize}

For sharpness consider the graph $T_1^-$ constructed as follows --- we take a path on $2k+1$ vertices for $k \geq 2$ and we add a leaf to vertex $v_k$ and to vertex $v_{2k}$.  Then $mp(T_1^-)=k$, while if $v$ is the leaf attached to $v_k$, $mp(T_1^--v)=2k= 2mp(T_1^-)$.  Figure \ref{diag4} shows the construction for $k=4$.   

\begin{figure}[h!]
\centering
\includegraphics{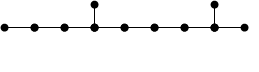} 
\caption{$T_1^-$  when $k=4$} \label{diag4}
\end{figure}
\medskip

We now consider the lower bounds.
\medskip

\noindent 1. \indent  Assume $mp(T+v) < \frac{mp(T)}{2}$.  Let $T'=T+v$ hence $T=T'-v$ --- then $mp(T') < \frac{mp(T'-v)}{2}$ implying that $mp(T'-v) > 2mp(T')$, contradicting the upper bound in part 2.

For sharpness, let $T_2^+$ be the path on $2k+1$ vertices, $k \geq 2$, with a leaf added to vertex $v_{2k}$.  Then $mp(T_2^+)=2k$.  We add a vertex $v$ and connect it to the vertex $v_{k}$ --- then $mp(T_2^++v)=k=\frac{mp(T_2^+)}{2}$.

\medskip

\noindent 2. \indent  Assume $mp(T-v) < \frac{mp(T)}{2}$.  Let $T'=T-v$ hence $T=T'+v$ --- then $mp(T') < \frac{mp(T'+v)}{2}$ implying $mp(T'+v) > 2mp(T')$, contradicting the upperbound in part 1.

For sharpness, we construct the graph $T_2^-$ as follows:  take the path on $2k+1$ vertices for $k \geq 2$ and add a leaf to every vertex except the first and the last so that $mp(T_2^-)=2k$.  Consider the leaf $v$ connected to $v_{k+1}$ --- if we delete this vertex, $mp(T_2^--v)=k=\frac{mp(T_2^-)}{2}$.

\end{proof}

\section{Operations involving two graphs}
\subsection{Cartesian Product}

The Cartesian product $G \cartprod H$ of graphs $G$ and $H$ is a graph such that the vertex set of $G \cartprod H$ is the Cartesian product $V(G) \times V(H)$ and any two vertices $(u,u')$ and $(v,v')$ are adjacent in $G \cartprod H$ if and only if either
\begin{itemize}
\item{$u = v$ and $u'$ is adjacent with $v'$ in $H$, or}
\item{$u' = v'$ and $u$ is adjacent with $v$ in G.}
\end{itemize}

\begin{theorem}
Let $G$ and $H$ be two connected graphs.  Then \[ mp(G)+mp(H)-1 \leq mp(G \cartprod H)  \leq mp(G)mp(H)\] and both bounds are sharp.
\end{theorem}

\begin{proof}
 Let us first consider the lower bound.  Let $v_1 \ldots v_t$ be a longest degree monotone path in $G$, and let $u_1 \ldots u_s$ be a longest degree monotne path in $H$, both in non-decreasing order.

Consider the path in $G \cartprod H$ with vertex coordinates \[(v_1,u_1)(v_1,u_2)\ldots (v_1,u_s)(v_2,u_s)(v_3,u_s)\ldots (v_t,u_s).\]  This is clearly a degree monotone path in $G \cartprod H$ with $t+s-1$ vertices, and hence $mp(G \cartprod H) \geq mp(G) + mp(H) -1$.

Now let us consider the degree monotone path in in $G \cartprod H$ of maximum length $r= mp(G \cartprod H)$ with vertices $z_i$ for $1 \leq i \leq r$.  Let us label $z_i=(v_{a_i},u_{b_i})$.  So consider the vertices from $z_1$ to $z_i=(v_{a_i},u_{b_i})$: for vertex  $z_{i+1}$, either the $v$ coordinate or the $u$ coordinate will change but not both.

If the $v$ coordinate changes, then $z_{i+1}=(v_{a_{i+1}},u_{b_{i+1}})$ where $a_{i+1}>a_i$ and $b_{i+1}=b_i$, hence it follows that $deg(v_{a_{i+1}}) \geq deg(v_{a_i})$.

If the $u$ coordinate changes, then $z_{i+1}=(v_{a_{i+1}},u_{b_{i+1}})$ where $a_{i+1}=a_i$ and $b_{i+1}>b_i$, hence it follows that $deg(u_{b_{i+1}}) \geq deg(u_{b_i})$.

Now let us consider those vertices in which the index $a_{i+1}>a_i$, which implies that the corresponding vertices in $G$ are distinct --- it is clear that these vertices form a degree monotone paths in $G$.  Hence if the number of such vertices is $t$, $t \leq mp(G)$.

Similiarly, if we consider the vertices in which the index $b_{i+1}>b_i$, the corresponding vertices in $H$ are distinct and form a degree monotone path in $H$ --- if the number of such vertices is $s$, then $s \leq mp(H)$.

Now since for each move from $z_i$ to $z_{i+1}$, only one coordinate changes, we have at most $st$ coordinates, hence $mp(G \cartprod H) \leq mp(G)mp(H)$.

Now we look at construction which achieve these bounds.  Firstly, let $G=H=K_{1,m}$, where $m \geq 2$, and hence $mp(G)=mp(H)=2$. In $G \cartprod H$ there is one vertex of degree $2m$, $2m$ vertices of degree $m+1$ which are independent, and $m^2$ vertices of degree 2 which are independent.  Hence the longest degree monotone path has three vertices so  $mp(G \cartprod H) = 2= mp(G) + mp(H) -1$.

Now for the upper bound, let $G$ be a connected regular graph on $t$ vertices and let $H$ be a graph such that $mp(H)=s$.  Let the vertices of $G$ be $v_1,\ldots, v_t$, and let the vertices $u_1,\ldots,u_s$ in $H$ be vertices on a longest degree monotone path in $H$.

Now in $G \cartprod H$, consider the path \[(v_1,u_1),(v_2,u_1),\ldots,(v_t,u_1),(v_t,u_2),(v_{t-1},u_2),\ldots,(v_1,u_2),(v_1,u_3),\ldots(v_t,u_3),\ldots \] and we carry on in this fashion until we have used all the $st$ vertices.  This path is degree monotone and hence $mp(G \cartprod H) = mp(G)mp(H)$.
\end{proof}

\subsection{Graph join}

The \emph{graph join} $G+H$ of two graphs $G$ and $H$ with disjoint vertex sets, $V(G)$ and $V(H)$ and disjoint edge sets $E(G)$ and $E(H)$ is the graph such that
\begin{itemize}
\item{$V(G+H)=V(G) \cup V(H)$}
\item{$E(G+H)=E(G) \cup E(H) \cup \{xy: x \in V(G), y \in V(H)\}$}
\end{itemize}

We now consider degree monotone paths in $G+H$.

\begin{theorem}
Given two graphs $G$ and $H$, \[mp(G)+mp(H) \leq mp(G+H) \leq |V(G)| + |V(H)|,\] and both bounds are sharp.
\end{theorem}

\begin{proof}
The upper bound is trivial since for any graph $mp(G) \leq |V(G)|$.  So let us consider the lower bound.  Let $P=v_1,v_2,\ldots,v_t$ be a degree monotone path of maximum length in $G$ and let $P^*=u_1,u_2,\ldots,u_s$ be a degree monotone path of maximum length in $H$.  Let us rearrange $\{v_1,\ldots,v_t,u_1,\ldots,u_s\}$ in non-decreasing order according to their degrees in $G+H$, noting that $deg_{G+H}(v)=deg_G(v)+|V(H)|$ while $deg_{G+H}(u)=deg_H(u)+|V(G)|$.  Then it is clear that these vertices form a degree monotone path in $G+H$ of length $s+t$, hence $mp(G+H) \leq mp(G) + mp(H)$.

Let us consider the upper bound.  Let $G$ and $H$ be two graphs such that $|V(G)|=|V(H)|$ and $G$ and $H$ have the same degree sequence.  We write the vertices of $G+H$ in non-decreasing order of their degrees such that each vertex of the $G$ part is followed by the corresponding vertex in the $H$ part.  So for the degree monotone path in $G+H$, we start with the vertex of smallest degree in $G+H$ and alternately take vertices of this same degree from $H$ and $G$ until all vertices of this degree are included in the path:  we then move to the second smallest degree in $G$ and carry out the same procedure for every different degree in the sequence.  There is an even number of vertices in $G+H$ of each degree since $G$ and $H$ have the same number of vertices and the same degree sequence, so this alternating path can be continued until all vertices in $G+H$ have been included, which implies that $mp(G+H)=|V(G)|+|V(H)|$ in this case, achieveing the upper bound.

For the lower bound, consider $G=K_{1,m}$ and $H=K_k$.  So $mp(G)=2$ and $mp(H)=k$.  Then in $G+H$ there are $k+1$ vertices of degree $k+m$, namely the vertices of $H$ and the vertex of degree $m$ in $G$, and the  $m$ vertices of degree $k+1$ that are independent.  So if we can start the path with a vertex of degree $k+1$, and then we must move to a vertex of degree $k+m$, and all the vertices of this degree can be included in the path.  hence $mp(G+H)=1 + k+1=2 +k = mp(G)+mp(H)$.
\end{proof}

\section{Conclusion}

We have given sharp bounds for $mp(G')$ in terms of $mp(G)$ where $G'$ is obtained from $G$ by the most basic operations involving a single vertex or edge.  We have shown that the effect of edge contraction on $mp(G)$ in the case of $K_3$-free graphs  is bounded (above and below) by a multiplicative factor, while there exist $K_4$-free graphs for which $mp(G)=4$ while $mp(G \cdot e)=|V(G)|-1$.  This leads to the following questions:
\begin{enumerate}
\item{Is there a $K_4$-free graph $G$ such that $mp(G)=3$ and $mp(G \cdot e) \geq |V(G)|-1$?}
\item{Is there a characterization of $K_4$-free graphs for which $mp(G \cdot e)$ is bounded above and below by a multiplicative factor?}
\end{enumerate}

We have also obtained sharp bounds for $mp(G \times H)$ in terms of $mp(G)$ and $mp(H)$ where $\times$ is either the Cartesian product of the join of two graphs.  Repeating this for other products, such as the direct product, might be interesting.

Finally, the notion of edge addition gives rise to the question of what the minimum number of edges of a graph $G$ on $n$ vertices can be if adding any edge increses $mp(G)$.  This leads to a problem analogous to the saturation number of a graph \cite{faudree2011survey,kaszonyi1986saturated}, and we shall be considering this in a forthcoming paper.

\bibliographystyle{plain}
\bibliography{dmp2}
\end{document}